\documentclass[12pt,letterpaper]{amsart}

\usepackage[top=2.2cm, bottom=2.2cm, left=3.2cm, right=3.2cm]{geometry}
\newtheorem{theorem}{Theorem}[section]
\newtheorem{lemma}[theorem]{Lemma}

\usepackage{amsmath}
\usepackage{amsfonts}
\usepackage{amssymb}

\newtheorem{remark}[theorem]{Remark}

\numberwithin{equation}{section}

\begin{document}

\title{Composition operators and Rational Inner Functions II: Boundedness between two different Bergman spaces.}

\author{Athanasios Beslikas} \footnote{ The author was partially supported by the National Science Center, Poland, SHENG III,
research project 2023/48/Q/ST1/00048}

\begin{abstract} In this note we provide a sufficient condition on when the composition operator $C_{\Phi}:A^2_{a}(\mathbb{D}^2)\to A^2_{\beta}(\mathbb{D}^2)$ is bounded, whenever $a\ge-1$ and $\beta$ is positive, with the assumption that $\Phi$ is induced by non-smooth Rational Inner Functions.  
\end{abstract}
\maketitle
\section{Introduction} Consider the weighted Bergman spaces $$A^2_{\beta}(\mathbb{D}^2)=\Biggl\{f\in \mathcal{O}(\mathbb{D}^2,\mathbb{C}):\int_{\mathbb{D}^2}|f(z_1,z_2)|^2(1-|z_1|^2)^{\beta}(1-|z_2|^2)^{\beta}dV(z_1,z_2)<+\infty\Biggr\},$$
where $\beta \ge -1.$ Let $\Phi=(\phi,\psi),$ where $\Phi\in\mathcal{O}(\mathbb{D}^2,\mathbb{D}^2)$ and assume that $\phi,\psi\in \mathcal{O}(\mathbb{D}^2,\mathbb{D})$ are induced by Rational Inner Functions. In this paper we study the action of the composition operator $C_{\Phi}(f)=f\circ\Phi$ induced by such symbols $\Phi$ on the weighted Bergman spaces $A^2_{\beta}(\mathbb{D}^2).$ The motivation lies in the fact that, in general, Rational Inner Functions may not be smooth on $\overline{\mathbb{D}^2},$ a fundamental difference to the studies of Bayart and Kosi\'nski in the papers \cite{Bayart2}, \cite{Bayart3} and \cite{Kosinski}. There, both authors considered symbols $\Phi$ which were $\mathcal{C}^2-$smooth on the closure of the bidisc.\\\\
Let $p\in \mathbb{C}[z_1,z_2]$ with bidegree $\mathrm{deg}(p)=(n,m)\in\mathbb{Z}_{>0}^2$ be a stable polynomial on $\mathbb{D}^2$ which might vanish at $\mathbb{T}^2.$ Then, according to Rudin's Theorem (see \cite{Rudin}) Rational Inner Functions on the bidisc setting are of the form
$\phi(z_1,z_2)=z_1^Nz_2^M\frac{\widetilde{p}(z_1,z_2)}{p(z_1,z_2)},$
where $z=(z_1,z_2)\in\mathbb{D}^2,$ and $\widetilde{p}(z_1,z_2)=z_1^nz_2^m\overline{p\left(\frac{1}{\bar{z_1}},\frac{1}{\bar{z_2}}\right)}.$ In contrast to the one dimensional setting, where the only Rational Inner Functions are the Finite Blaschke Products, in two complex dimensions, Rational Inner Functions may have singularities which occur on the distinguished boundary $\mathbb{T}^2.$ By a result of Knese, we also know that Rational Inner Functions have non-tangential limit everywhere on $\mathbb{T}^2,$ even on the singularities that might occur, and the value of this non-tangential limit is unimodular. A concrete example of a Rational Inner Function on the bidisc is the Knese Function, 
$$\varphi(z)=\frac{2z_1z_2-z_1-z_2}{2-z_1-z_2}$$
for $(z_1,z_2)\in\mathbb{D}^2.$ We observe that this function has a singularity on the point $(1,1)\in \mathbb{T}^2.$
For more about Rational Inner Functions and their properties the reader can consult the works \cite{Bickel1}, \cite{Bickel2}, \cite{Bickel3},  \cite{Bickel4}, \cite{Knese1} and \cite{Knese2} and the references therein.
\section{Statement and proof of the main result}
In the paper \cite{Me} an initial investigation of the composition operators induced by Rational Inner Functions is made. There, it is investigated how the composition operator, acting on the Bergman space of the bidisc, behaves for specific symbols of the form $\Phi=(\phi,\phi).$ 
In the present note we provide a boundedness result for all self-maps of the bidisc induced by RIFs of more general form $\Phi=(\phi,\psi)$, where $\phi,\psi$ are not necessarily the same Rational Inner Function, but they have both one singularity on $\mathbb{T}^2.$  This boundedness result is for the composition operator $C_{\Phi}:A^2_{a}(\mathbb{D}^2)\to A_{\beta}^2(\mathbb{D}^2)$ where $a,\beta$ are positive exponents different from each other. In order to prove boundedness, one needs the following crucial lemma which is a Carleson measure criterium for the pull-back measure. 
\begin{lemma}
Let $\Phi:\mathbb{D}^2\to \mathbb{D}^2$ be a holomorphic self-map of the bidisc. The composition operator $C_{\Phi}: A^2_{a}(\mathbb{D}^2)\to A^2_{\beta}(\mathbb{D}^2)$ is bounded if and only if there is a constant $C>0$ such that for every $\tilde{\delta} \in (0,2)^2$ and $\zeta\in\mathbb{T}^2$:
$$V_{\beta}(\Phi^{-1}(S(\zeta,\tilde{\delta})))\leq CV_{a}(S(\zeta,\tilde{\delta})).$$
\end{lemma}
Recall that a two-dimensional Carleson box is defined as $$S(\zeta,\tilde{\delta})=\{(z_1,z_2)\in\mathbb{D}^2:|z_1-\zeta_1|<\delta_1,|z_2-\zeta_2|<\delta_2\},$$
where
$\zeta=(\zeta_1,\zeta_2)\in \mathbb{T}^2, \tilde{\delta}=(\delta_1,\delta_2)\in(0,2)^2.$
Moreover, its volume behaves like $V_{a}(S(\zeta,\delta))\asymp\delta_1^{a+2}\delta_2^{a+2}.$
The above lemma stated in \cite{Koo} was the main tool in the papers of  Kosi\'nski \cite{Kosinski}, and Bayart \cite{Bayart2} \cite{Bayart3}. We shall also need the following Lemmata.
\begin{lemma}
Let $p\in\mathbb{C}[z_1,z_2]$ be stable in $\mathbb{D}^2$ and vanish in at least one point $\tau\in\mathbb{T}^2$. Then, the intersection of the zero set of the polynomial $P_{\zeta}(z_1,z_2)=\widetilde{p}(z_1,z_2)-\zeta p(z_1,z_2)$ and $\mathbb{D}^2$ is empty $(\mathcal{Z}(P_{\zeta})\cap\mathbb{D}^2=\varnothing)$ and $\mathcal{Z}(P_\zeta)\cap\overline{\mathbb{D}^2}\subset\partial\mathbb{D}^2$ for all $\zeta\in\mathbb{T}.$
\label{zeroset}
\end{lemma}
\begin{proof} Fix $\zeta\in\mathbb{T}$ for a moment and let $p\in\mathbb{C}[z_1,z_2]$ and consider the polynomial $P_{\zeta}(z_1,z_2)=\widetilde{p}(z_1,z_2)-\zeta p(z_1,z_2).$ Let $\phi=\frac{\widetilde{p}}{p}$ the RIF induced by $p.$ By the Maximum Modulus Principle, we know that $|\phi|<1$ in $\mathbb{D}^2$ and $|\phi|=1$ on $\mathbb{T}^2.$ Thus, $P_{\zeta}(z_1,z_2)=0$
if and only if $p(z_1,z_2)(\phi(z_1,z_2)-\zeta)=0$ and $p$ vanishes only on $\mathbb{T}^2.$ Now $p$ does not vanish on $\mathbb{D}^2$ and $|\phi|<1$ on $\mathbb{D}^2,$ hence the only way for $P_{\zeta}(z_1,z_2)=0$ is whenever $\phi=\zeta.$ This cannot happen on the open bidisc $\mathbb{D}^2.$ Moreover, $\widetilde{p},p$ vanish both at $\tau\in\mathbb{T}^2$, so $\{\tau\}\subset \mathcal{Z}(P_{\zeta})$ which is Lemma 2.8 in \cite{Pascoe}, a consequence of the "Edge of the Wedge" Theorem in particular. Moreover, in \cite{Pascoe}, it is shown that $P_{\zeta}$ has no zeros on $\mathbb{D}^2\cup \mathbb{E}^2,$ where $\mathbb{E}$ is the exterior unit disc $\mathbb{C}\setminus\overline{\mathbb{D}}.$ Hence, the only possibility is that $\mathcal{Z}(P_{\zeta})\subset(\mathbb{T}\times\mathbb{D})\cup(\mathbb{D}\times\mathbb{T})\cup\mathbb{T}^2.$  
\end{proof}
\begin{remark}
Note here that as long as $\mathcal{Z}(\widetilde{p}-\zeta p)\subset \partial \mathbb{D}^2,$ then $\mathrm{dist}(z,\mathcal{Z}(\widetilde{p}-\zeta p))\ge \min\{(1-|z_1|),(1-|z_2|)\}.$ Moreover, if $z=(z_1,z_2)$ lies close to the topological boundary $\partial\mathbb{D}^2,$ then $(1-|z_i|)\asymp(1-|z_i|^2),$ for $i=1,2.$ This simple observation will be useful later on.
\end{remark}
Moreover, we need the next lemma.
\begin{lemma} Let $(z_1,z_2)\in\mathbb{D}^2.$ Then, for all $\beta>0$ and $\delta\in(0,1)$
$$V_{\beta}(\{z\in\mathbb{D}^2:(1-|z_1|^2)(1-|z_2|^2)\leq \delta\})\simeq\delta^{\beta+1}\log\frac{1}{\delta}.$$
The symbol "$\simeq$" means that the quantities on the left- and right-hand sides are analogous to each other. The implied positive constant depends only on $\beta$.
\label{vollema}
\end{lemma}
\begin{proof} Set $u=1-|z_1|^2$ and $v=1-|z_2|^2$. It suffices to calculate the integral $\iint_{\{uv\leq \delta, 0<u,v<1\}}u^{\beta}v^{\beta}dudv.$ A routine calculation shows that this volume behaves like $\delta^{\beta+1}\log\frac{1}{\delta}.$
\end{proof}
Let us now  give the precise statement of the main result.
\begin{theorem} Let $\Phi=(\phi,\psi)$ with $\phi=\frac{\tilde{p_1}}{p_1}$ and $\psi=\frac{\tilde{p_2}}{p_2}.$ Assume that both polynomials $p_1,p_2$ have one zero on $\mathbb{T}^2$. Then there exists a $q>0,$ such that the composition operator $C_{\Phi}:A^2_{\frac{\beta}{2q}-2}(\mathbb{D}^2)\to A^2_{\beta}(\mathbb{D}^2)$ is bounded for all $\beta>2q.$
\label{bounded}
\begin{proof} We need to estimate the volume $V_{\beta}(\Phi^{-1}(S(\zeta,\tilde{\delta})))$ for all $\tilde\delta\in(0,2)^2$ and for all $(\zeta_1,\zeta_2)=\zeta \in \mathbb{T}^2.$ Similarly to \cite{Kosinski}, we can assume $\zeta_1=1.$ We establish first a local estimate for $|\phi(z)-1|. $ Take $\eta\in\mathbb{T}^2$ such that $\phi^{*}(\eta)=1,$ and $p_1(\eta)=0,$ where $\phi^{*}(\eta)$ denotes the non-tangential value of $\phi$ at the singularity $\eta.$ Consider a neighborhood $\mathcal{U}_1(\eta)$ of $\eta\in\mathbb{T}^2$ intersecting the interior of bidisc. By applying Łojasiewicz inequality on $\overline{\mathcal{U}_1(\eta)\cap\mathbb{D}^2}$ (see \cite{Loja} for a precise statement of the inequality) we find an exponent $q_1>0$ and a constant $C_1>0$ such that the inequality $|P_{\zeta}(z)|=|\widetilde p_1-1\cdot p_1|\ge C_1\mathrm{dist}^{q_1}(z,\mathcal{Z}(P_{\zeta})),$ holds for $z\in \overline{\mathcal{U}_1(\eta)\cap\mathbb{D}^2},$ and for $\phi^*(\eta)=\zeta_1=\zeta=1\in\mathbb{T}.$ Therefore, 
after applying Lemma \ref{zeroset}, one receives the chain
\begin{align} |\phi(z)-1|=&\left|\frac{\widetilde{p_1}(z)-p_1(z)}{p_1(z)}\right|&&\\
\nonumber
\ge&\frac{C_1\mathrm{dist}^{q_1}(z,\mathcal{Z}(\tilde{p_1}- p_1))}{M_1}&&\\
\nonumber \ge&\frac{C_1\mathrm{min}^{q_1}\{(1-|z_1|^2),(1-|z_2|^2)\}}{M_1}&&\\
\nonumber \ge&\frac{C_1(1-|z_1|^2)^{q_1}(1-|z_2|^2)^{q_1}}{M_1},
\end{align}
for all $z\in \overline{\mathcal{U}_1\cap\mathbb{D}^2}$, where $M_1=\max\{|p_1(z)|,z\in\overline{\mathcal{U}_1\cap\mathbb{D}^2}\}$ of $|p_1(z)|.$ Note that $\overline{\mathcal{U}_1\cap\mathbb{D}^2}$ is a small compact set around $\eta\in\mathbb{T}^2$ for which the \L{}ojasiewicz inequality holds.
By invoking Lemma \ref{vollema} we observe that
\begin{align} 
V_{\beta}(\Phi^{-1}(S(\zeta,\widetilde{\delta}))\cap \mathcal{U}_1)\leq& V_{\beta}(\{z\in\mathbb{D}^2\cap\mathcal{U}_1:|\phi(z)-1|<\delta_1\})&&\\
\nonumber\leq&\int_{\{z\in\mathbb{D}^2\cap\mathcal{U}_1:(1-|z_1|^2)(1-|z_2|^2)\leq(\frac{ M_1}{C_1}\delta_1)^{1/q_1}\}} dV_{\beta}(z_1,z_2)&&\\
\nonumber\lesssim& \delta_1^{\frac{\beta+1}{q_1}}\log\frac{1}{\delta_1}.
\end{align}
By rotating $\psi$ and composing it to a rotation to bring the singularity on the point $\eta\in\mathbb{T}^2,$ we repeat the same volume estimates for $\psi.$ One then, obtains a neighborhood $\mathcal{U}_2(\eta),$ an exponent $q_2>0$ and a positive constant $C_2>0$ such that
\begin{align} 
V_{\beta}(\Phi^{-1}(S(\zeta,\widetilde{\delta}))\cap\mathcal{U}_2)\leq& V_{\beta}(\{z\in\mathbb{D}^2\cap\mathcal{U}_2:|\psi(z)-1|<\delta_2\})&&\\
\nonumber\leq&\int_{\{z\in\mathbb{D}^2\cap\mathcal{U}_2:(1-|z_1|^2)(1-|z_2|^2)\leq (\frac{M_2}{C_2}\delta_2)^{1/q_2}\}} dV_{\beta}(z_1,z_2)&&\\
\nonumber\lesssim& \delta_2^{\frac{\beta+1}{q_2}}\log\frac{1}{\delta_2}.
\end{align}
Note that the volume measure is invariant under such rotations.
At this point one has to absorb the logarithmic factor in both estimates. To achieve this, we observe that there exists $\delta_0>0$ such that for all $\delta<\delta_0$ the inequality $\log\frac{1}{\delta}\leq \frac{1}{\delta^{\epsilon}}$ holds for all fixed $\epsilon>0.$ Consider now as $\mathcal{U}=\mathcal{U}_1\cap\mathcal{U}_2$ the intersection of the two neighborhoods found above. Arguing as in the proof of Theorem 10 in \cite{Kosinski}, it is enough to show $V_{\beta}(\Phi^{-1}(S(\zeta,\widetilde{\delta}))\cap \mathcal{U})\leq C\delta_1^{a+2}\delta_2^{a+2}$ for $a=\frac{\beta}{2q}-2$ and for some $C>0.$ Choosing $\epsilon=\frac{1}{q}>0,$  
multiplying the two inequalities by parts (after shrinking the neighborhoods if necessary to make the two inequalities hold in the same domain) and taking the square root, yields
$$V_{\beta}(\Phi^{-1}(S(\zeta,\widetilde{\delta}))\cap\mathcal{U})\leq C(q_1,q_2,\beta)\delta_1^{\frac{\beta+1}{2q_1}}\delta_2^{\frac{\beta+1}{2q_2}}\sqrt{\log\frac{1}{\delta_1}\log\frac{1}{\delta_2}}\leq C(q_1,q_2,\beta)\delta_1^{\frac{\beta}{2q}}\delta_2^{\frac{\beta}{2q}},$$ for all $\delta_1,\delta_2$ sufficiently small, where $q=\max\{q_1,q_2\}$ and
 $C(q_1,q_2,\beta),$ is a positive constant depending only on $q_1,q_2,\beta>0.$ The proof only works whenever $\frac{\beta}{2q}-2>-1.$ Thus, $\beta>2q.$ This finishes the proof. 
\end{proof}
\end{theorem}
\section{Comments}
Let us comment on the result above. As one can observe, the stated result is consistent with Example 5.1. in the paper \cite{Me}. In that specific example, we considered the symbol
 $$\Phi=(\varphi,\varphi)=\left(\frac{2z_1z_2-z_1-z_2}{2-z_1-z_2},\frac{2z_1z_2-z_1-z_2}{2-z_1-z_2}\right), (z_1,z_2)\in \mathbb{D}^2,$$ and proved that $C_{\Phi}:A^2_{\frac{\beta}{4}-2}(\mathbb{D}^2)\to A^2_{\beta}(\mathbb{D}^2)$ is bounded. Here, in the setting of Theorem \ref{bounded} the exponent $q$ satisfies $q=2.$ In the same spirit, we can also consider selfmaps of the form
 $$\Phi_{A,B}(z_1,z_2)=\left(\frac{2z_1z_2-z_1-z_2}{2-z_1-z_2},\frac{2z_1z_2-\overline{B}z_1-\overline{A}z_2}{2-Az_1-Bz_2}\right)=\left(\frac{\widetilde{p}_1}{p_1},\frac{\widetilde{p_{A,B}}}{p_{A,B}}\right),(z_1,z_2)\in\mathbb{D}^2$$
 for $A,B\in\mathbb{T}\setminus\{1\}$ satisfying $|A|+|B|=2.$ Both co-ordinate functions are induced by polynomials of bidegree $(1,1).$ After a rotational argument we can assume that the \L{}ojasiewicz exponent for the $|\widetilde{p_{A,B}}-p_{A,B}|$ will be equal to the one of $|\widetilde{p}-p|.$ This implies that $C_{\Phi_{A,B}}:A^2_{\frac{\beta}{4}-2}(\mathbb{D}^2)\to A^2_{\beta}(\mathbb{D}^2).$
 Nevertheless, one has to note that  
 the technique was different in Example 5.1. The main idea there was to consider the Sum of Squares formula for $p(z)=2-z_1-z_2$, with $(z_1,z_2)\in\mathbb{D}^2$ to obtain estimates on $|\varphi(z)-\zeta|.$ 
 Another fact that needs to be pointed out is that the volume estimates here are not sharp. Moreover, the exponents $\beta$ are positive. To this end, we propose the following interesting open problem.\\\\
\textbf{Problem 1:} \textit{Let $\Phi=(\phi,\psi)$ a holomorphic self map of the bidisc, where $\phi,\psi$ are RIFs with singularities. Characterize the RIFs that induce bounded composition operators acting on $A^2_{\beta}(\mathbb{D}^2),$ for all $\beta\ge-1.$}\\\\

The difficulty, of course, rises from the fact that RIFs with singularities might be non-tangentially smooth at best (see the work of Knese in \cite{Knese2}) and volume estimates in the spirit of Bayart and Kosi\'nski are not easily attainable. Nevertheless, giving such a characterization for RIFs might be the appropriate way to approach results more general than the results of \cite{Bayart2}, \cite{Kosinski} and \cite{Koo} for holomorphic self-maps of the bidisc which are not necessarily smooth.
\section{Acknowledgements} Part of the preparation of this work was conducted during the "1st Conference of Mathematics-Deskati 20th-23rd of August 2025". I would like to thank Athanasios Kouroupis, Dimitrios Vavitsas and municipality of Deskati for the excellent staying and working conditions. Moreover, I express my gratitude to prof. Alan Sola and prof. Kelly Bickel for their time discussing this note.

\bibliographystyle{amsplain}

Athanasios Beslikas,\\
Doctoral School of Exact and Natural Studies\\
Institute of Mathematics,\\
Faculty of Mathematics and Computer Science,\\
Jagiellonian University\\ 
\L{}ojasiewicza 6\\
PL30348, Cracow, Poland\\
athanasios.beslikas@doctoral.uj.edu.pl
\end{document}